\documentclass[reqno, english, 12pt]{amsart}
\pdfoutput = 1 
\usepackage[T1]{fontenc}

\usepackage[dvipsnames]{xcolor} 
\usepackage{amsfonts,amsmath,amssymb,amsthm,bbm,mathtools,comment}

\usepackage[shortlabels]{enumitem}
\usepackage{cancel}
\usepackage[hyphens]{url} 
\usepackage[linktoc = all, hidelinks, colorlinks, unicode=true, pagebackref=true]{hyperref} 
\usepackage[capitalise, compress, nameinlink, noabbrev]{cleveref} 
  \hypersetup{linkcolor={blue!70!black}, citecolor={green!70!black}, urlcolor={blue!70!black}}

\usepackage[mathscr]{euscript}
\usepackage{adjustbox,tikz,calc,graphics,babel,standalone}
\usetikzlibrary{decorations.pathmorphing,shapes,fit,positioning}
\usetikzlibrary{shapes.misc,calc,intersections,patterns,decorations.pathreplacing}
\usetikzlibrary{arrows,shapes,positioning}
\usetikzlibrary{decorations.markings,quotes,arrows.meta}
\usepackage[final]{microtype}
\usepackage[numbers]{natbib}
\usepackage{cmtiup}
\usepackage{graphicx}
\usepackage{caption}
\usepackage{subcaption}
\usepackage{verbatim}
\usepackage{array}
\usepackage[frame,cmtip,arrow,matrix,line,graph,curve]{xy}
\usepackage{pifont}
\usepackage{todonotes}
\usetikzlibrary{topaths,calc} 
\tikzstyle{vertex} = [fill,shape=circle,node distance=80pt]
\tikzstyle{edge} = [opacity=0.4,fill opacity=0.0,line cap=round, line join=round, line width=40pt]
\tikzstyle{elabel} =  [fill,shape=circle,node distance=30pt]
\tikzstyle{circ} = [draw, fill=SpringGreen, circle, inner sep=0.12cm]
\tikzstyle{square} = [draw, fill=RedOrange, rectangle, inner sep=0.15cm]
\tikzstyle{triangle} = [draw, fill=Turquoise, regular polygon, regular polygon sides=3, inner sep=0.08cm]

\usepackage[margin = 2.3cm,foot=1cm]{geometry}

\allowdisplaybreaks

\theoremstyle{plain}
\newtheorem{theorem}{Theorem}[section]		
\newtheorem{lemma}[theorem]{Lemma}
\newtheorem{claim}[theorem]{Claim}

\newtheorem{corollary}[theorem]{Corollary}

\theoremstyle{remark}

\crefname{theorem}{theorem}{theorems}
\Crefname{theorem}{Theorem}{Theorems}

\crefname{lemma}{lemma}{lemmas}
\Crefname{lemma}{Lemma}{Lemmas}

\crefname{claim}{claim}{claims}
\Crefname{claim}{Claim}{Claims}

\crefname{proposition}{proposition}{propositions}
\Crefname{proposition}{Proposition}{Propositions}

\crefname{corollary}{corollary}{corollaries}
\Crefname{corollary}{Corollary}{Corollaries}

\crefname{conjecture}{conjecture}{conjectures}
\Crefname{conjecture}{Conjecture}{Conjectures}

\crefname{problem}{problem}{problems}
\Crefname{problem}{Problem}{Problems}

\crefname{observation}{observation}{observations}
\Crefname{observation}{Observation}{Observations}

\crefname{example}{example}{examples}
\Crefname{example}{Example}{Examples}

\crefname{question}{question}{questions}
\Crefname{question}{Question}{Questions}

\crefname{definition}{definition}{definitions}
\Crefname{definition}{Definition}{Definitions}

\crefname{construction}{construction}{constructions}
\Crefname{construction}{Construction}{Constructions}

\renewcommand{\emptyset}{\varnothing}
\renewcommand{\le}{\leqslant}
\renewcommand{\leq}{\leqslant}
\renewcommand{\ge}{\geqslant}
\renewcommand{\geq}{\geqslant}

\let\originalleft\left
\let\originalright\right
\renewcommand{\left}{\mathopen{}\mathclose\bgroup\originalleft}
\renewcommand{\right}{\aftergroup\egroup\originalright}

\makeatletter
\def\imod#1{\allowbreak\mkern10mu({\operator@font mod}\,\,#1)}
\makeatother

\author{
Ant\'onio Gir\~ao and Zach Hunter}
\thanks{
AG: Department of Mathematics, University College London, London, WC1E 6BT, UK.
E-mail: \texttt{a.girao@ucl.ac.uk}\\
ZH: Department of Mathematics, ETH, Z\"urich, Switzerland. Research supported in part by SNSF grant 200021-228014. Email: \texttt{zach.hunter@math.ethz.ch}.
}

\title{Induced subdivisions of $K_{d+1}$ in graphs of high girth}

\begin{document}
\maketitle
\begin{abstract}
    In this paper, we show that for all $k\geq 10^8$, every graph with minimum degree $k$ and girth at least $10^8$ contains an \textit{induced} subdivision of a $K_{k+1}$. This answers a problem asked by K\"uhn and Osthus (originally attributed to Shi).
\end{abstract}

\section{Introduction}

A central theme in extremal graph theory is concerned with showing that graphs with sufficiently high but constant average degree have a large complete graph in some sense (e.g., as a minor). For example, classical results of Koml\'os--Szemer\'edi and of Bollob\'as--Thomason show that average degree $\Theta(t^2)$ guarantees a subdivision (topological minor) of $K_t$~\citep{KomlosSzemeredi1994TopologicalCliques,BollobasThomason1998TopologicalCompleteSubgraphs} and this is tight up to the implied constant. 

Soon after these results, Mader asked whether this bound could be significantly improved for graphs with girth at least $6$. This was recently resolved by Liu and Montgomery~\citep{LiuMontgomery2017MaderC4Free}. Using the Koml\'os--Szemer\'edi framework of sublinear expanders, they proved that there is an absolute constant $C$ such that every such graph $G$ with average degree at least $Ck$ contains a subdivision of $K_k$.

It is then natural to wonder whether one can go even further; if it is true that every graph with minimum degree $k-1$ and large enough girth must contain a subdivision of a complete graph on $k$ vertices. Obviously if true this would be tight on the minimum degree condition. This problem was answered by Mader~\cite{Mader} and later the girth condition was improved by K\"uhn and Osthus~\citep{KuhnOsthus2006ImprovedBounds}. 

A high girth hypothesis is a natural way to force local sparsity and in~\citep{KuhnOsthus2006ImprovedBounds} the authors asked (see Problem~3, attributed originally to Shi) whether the same phenomenon holds for \textit{induced} subdivisions of large cliques. More precisely, they asked whether for every $r$, there is an $h(r)$ such that every graph with minimum degree at least $r$ and girth at least $h(r)$ contains an \textit{induced} subdivision of a $K_{r+1}$.

Our main result gives a positive answer in a strong sense, showing that $h(r)$ is bounded from above by an absolute constant. 

\begin{theorem}\label{thm: main}
    Let $G$ be a graph with minimum degree $d\geq 10^8$ and girth at least $10^8$. Then $G$ contains an induced subdivision of $K_{d+1}$.
\end{theorem}
\noindent We did not try to optimize either of the above constants, as doing so would make the paper unnecessarily more difficult to read. However, we believe much lower constants would be possible using the exact same techniques.

\section{Preliminaries}
In this section, we collect some well-known results, along with some easy lemmas. 

\begin{lemma}[Lov\'asz Local Lemma~\citep{ErdosLovasz1975LLL,AlonSpencerProbMethod}]\label{lem: lll}
    Let $(\Omega, \mathbb{P})$ be a probability space and $A_1,\ldots, A_t$ be events where $\mathbb{P}[A_i]\leq p$. Moreover, suppose that for every $i\in [t]$, $A_i$ is independent of all events except at most $d$. Then, provided $p\leq \frac{1}{e(d+1)}$, we have $$\mathbb{P}\bigl[\cap_{i\in [t]}\overline{A_i}\bigr]>(1-ep)^t.$$
\end{lemma}
\begin{lemma}(see e.g. ~\citep{AlonHooryLinial2002MooreIrregular})\label{lem : girth}
    For every $g,d\geq 3$, any graph with minimum degree $d$ and girth at least $g$ must have at least $d^{g/2}$ vertices.
\end{lemma}

Several times, we shall black-box the following result which finds \textit{non-induced} subdivisions.
\begin{theorem}\label{thm: subdivision}
    Every graph with average degree $10d^2$ contains a subdivision of a $K_{d+1}$~\citep{BollobasThomason1998TopologicalCompleteSubgraphs,KomlosSzemeredi1994TopologicalCliques,LiuMontgomery2017MaderC4Free}.  
\end{theorem}
\noindent In one instance, we will further need a result\footnote{Technically speaking, we will quote the improved constant obtained by Thomas and Wollan \cite{WollanThomas}. But their original bound would also suffice for us.} of Bollob\'as and Thomason on $k$-linkedness \cite{BollobasThomason1998TopologicalCompleteSubgraphs}.
Given a graph $G$ on at least $2k$ vertices, we say $G$ is $k$-linked if for any two disjoint sets $\{x_1,\ldots x_k\},\{y_1,\dots,y_k\}$ of $k$ vertices, we can find vertex-disjoint paths $P_1,\dots,P_k$ where the end-points of $P_i$ are $x_i$ and $y_i$. 

\begin{lemma}\label{lem: linkedness}\cite{BollobasThomason1998TopologicalCompleteSubgraphs,WollanThomas}
    Let $G$ be a $10k$-connected graph then $G$ is $k$-linked. 
\end{lemma}\noindent Being $k$-linked is useful, as it allows one to find a subdivision of $K_{\Omega(k^{1/2})}$ in a robust way. In particular, up to constants, it allows one to quickly recover Theorem~\ref{thm: subdivision} (using a theorem of Mader that graphs with average degree $d$ we have as $d/4$-connected subgraph).

Finally, we need a nice result (see, e.g.,~\citep{PenevThomasseTrotignon2014Isolating}) which shows that every graph with high enough average degree contains a highly connected subgraph with ``small boundary''. It shall serve as induced analogue of Mader's result, which will allow us to apply Lemma~\ref{lem: linkedness} for the most technical part of our paper. 

Let $G$ be a graph and $X\subset V(G)$. We define the boundary of $X$ to be
\[
\partial(X)=\{x\in X: \text{ there exists } y\in V(G)\setminus X \text{ with } xy\in E(G)\}.
\]
\begin{lemma}\label{lem: boundary}
    Let $k\geq 2$ be a positive integer and $G$ be a graph with minimum degree at least $4k^2$. Then, $G$ contains a $k$-connected subgraph with more than $4k^2$ vertices such that the boundary has at most $2k^2$ vertices.
\end{lemma}

\section{}

The next lemma that allows us to find an induced subdivision of a clique in a very unbalanced bipartite graph.

\begin{lemma}\label{lem: unbalanced}
    Let $G$ be a graph with disjoint vertex sets $A,B$ satisfying the following for some integer $d\geq 3$: 
    \begin{enumerate}[(i)]
     \item $G$ is $2d$-degenerate; 
     \item $G$ has girth at least $5$;
        \item $|A|\geq 10^5 d^{4}|B|$; 
        \item $d_{B}(x)\geq 2$ for every $x\in A$.
    \end{enumerate}

    Then, $G$ contains an induced subdivision of $K_{d+1}$.
\end{lemma}

\begin{proof}
    Since $G$ is $2d$-degenerate, we can find $A'\subset A$ of size at least $\frac{|A|}{2d+1}\ge \frac{|A|}{3d}$ which forms an independent set. 
    Moreover, by degeneracy there are at most $|A'|/2$ vertices which send more than $4d$ edges to $B$. Let $A''$ be the set of vertices of $A'$ which send at most $4d$ edges to $B$, by assumption $|A''|\geq |A'|/2\geq |A|/6d$. 
    
    We now randomly choose a vertex in $B$ with probability $p=\frac{1}{10d}$ independently, denote the random subset $R$. Let $z_1,\ldots z_{|B|}$ be an ordering of the vertices of $B$ such that $z_i$ has at most $2d$ neighbours to $z_j$ with  $j> i$ and let $R'$ be the subset of $R$ where every vertex in $R'$ has no right neighbour in $R'$. It follows that $G[R']$ is an independent set.
    We say a vertex $y\in A''$ is \textit{good} if it has exactly two neighbours in $R'$.

It is not hard to see that $\mathbb{P}[y \text{ is good}]\geq p^2(1-p)^{4d}(1-p)^{4d}\geq p^2\cdot (1/2) \cdot (1/2)=\frac{1}{400d^2}$ (importantly $y$ has at least two neighbors $b_1,b_2$ in $B$, and they are not adjacent by our girth assumption; we then get this lower bound for the probability $N(y)\cap R'=\{b_1,b_2\}$). Let $X$ be the number of \textit{good} vertices. 
Hence $\mathbb{E}[X-10d^2|R'|]\ge |A|/2400d^3 -10d^2p|B|\geq 0$ and so there is a choice where $X\geq 10d^2|R'|$. Fix such a choice of $R'$. We construct an auxiliary graph $H$ with vertex set $R'$ where $xy\in E(H)$ iff there is $z\in A''$ with $N_{R'}(z)=\{x,y\}$. Note that if $H'$ is a subgraph of $H$, then the $1$-subdivision of $H'$ is an induced subgraph of $G$. Thus it remains to show $H$ has $K_{d+1}$ as an induced subgraph.

Now by our girth assumption, for every pair $x,y\in R'$ there is at most one vertex $z$ with $N_{R'}(z)=\{x,y\}$ (since $G$ is $C_4$-free). Therefore, by the choice of $R'$, we know $|E(H)|\geq 10d^2|R'|$ and thus Theorem~\ref{thm: subdivision} implies we can find a subdivision of $K_{d+1}$ as subgraph of $H$. As previously discussed, this implies $G$ has a $K_{d+1}$ subdivision as induced subgraph, completing the proof. 
\end{proof}

\begin{lemma}\label{lem: largesub}
    Let $d\geq 10^8$ and $G$ be a $2d$-degenerate graph on $n$ vertices with girth at least 5 and no induced subdivision of $K_{d+1}$. Moreover, let $X\subset V(G)$ where $|X|\geq n/2$ and every vertex in $X$ has degree at least $d$.
    Then, we can find the following structure.
Two pairwise disjoint sets $X', Y$ satisfying the following:
\begin{enumerate}
    \item $X'\subseteq X$ and $Y\subseteq (V(G)\setminus X')$;
    \item $|Y|\geq \frac{n}{3d^{7}}$;
    \item  $Y$ is an independent set;
    \item $2 \leq |N(y)\cap X'|\leq d^{6}$ for every $y\in Y$;
    \item $d(x')\leq 8d$ for all $x'\in X'$.
\end{enumerate}
\end{lemma}
\begin{proof}
    Firstly, by $2d$-degeneracy we may pass to a subset of $Z\subseteq X$ of size $n/4$ such that every vertex had degree at most $8d$. We uniformly at random split $Z$ into two parts $Z_1\cup Z_2$. 
    Note that the expected number of edges between $Z_1$ and $V(G)\setminus Z_1$ is at least $(n/4)\cdot d /2\geq nd/8$. 
    Fix such a partition.

    Let $U\subset V(G)\setminus Z_1 $ be the collection of vertices which send at least $2$ edges to $Z_1$. 
    It is clear that there are at most $n$ edges from $V(G)\setminus U$ to $Z_1$ since each such vertex sends at most one edge to $Z_1$. 
    Therefore, $e(Z_1,U)\geq nd/8 - n\geq nd/10$. 
    Finally, let $U'\subseteq U$ be the set of vertices which send more than $d^{6}$ edges to $Z_1$. By $2d$-degeneracy, we have that $|U'|\leq 2n/d^{5}$.

    Suppose for contradiction that $|U\setminus U'|\leq n/d^{6}$. Since the total number of edges from $U\setminus U'$ to $Z_1$ is at most $(n/d^6) \cdot d^6\leq n$ we must have at least $nd/10 -n \geq nd/12$ edges in $(Z_1, U\setminus U')$. Delete all vertices from $Z_1$ which send less than $2$ edges to $U\setminus U'$ and let $Z'_1$ be the remaining vertices. Note that $e(Z'_1, U\setminus U')\geq nd/12 - n\geq nd/14$. By assumption on the maximum degree of the vertices in $Z'_1$ we must have that $|Z'_1|8d\geq nd/14$ which implies $|Z'_1|\geq n/ 200$. We may now invoke Lemma~\ref{lem: unbalanced} with $Z'_1=:A$ and $U'=:B$ (note that $\frac{|A|}{|B|}\ge \frac{d^5}{400}\ge 10^5 d^4$ as we assumed $d\ge 10^8$), contradicting the assumption. 
    
    Therefore we can assume $|U\setminus U'|\geq n/d^6$, whence we are done by taking $X':=Z_1$ and $Y\subset U\setminus U'$ independent of size at least $n/3d^7$ (using that $\chi(G)\le 2d+1\le 3d$ by $2d$-degeneracy).
\end{proof}

\begin{lemma}\label{lem: connectedgood}
    Let $d\geq 5000, C
    \geq 1$ and $G$ be a graph on $n$ vertices with $\delta(G)\geq d^6$ and $\Delta(G)\leq d^{C}$,  Moreover, let $B\subset V(G)$ be a subset of size at least $\frac{n}{d^{C}}$. Then, provided $g(G)\geq 20C$, there exists a $100d^2$-connected subgraph $H'\subset G$ containing at least $\frac{|V(H')|}{2d^C}$ vertices $x \in B$ with such that $d_G(x)= d_{H'}(x)$. 
\end{lemma}
\begin{proof}
    Let $G_1 =G$. Given $t\ge 1$, if $G_t$ is a non-empty graph with minimum degree $d^5$, then by Lemma~\ref{lem: boundary} we can find an induced subgraph $H_t\subset G_t$ so that $H_t$ satisfies:
 \begin{itemize}
     \item $100d^2$-connected (with at least $3$ vertices\footnote{This is simply to sidestep conventions regarding the connectivity/girth of cliques on at most $2$ vertices.});
     \item has boundary $S_t := \partial_{G_t}(H_t)$ of size at most $400 d^4$.
 \end{itemize}
We fix one such $H_t$, delete $V(H_{t})$ and additionally delete vertices with degree less than $d^5$ in the current graph until no more such vertices remain. Call this set of additional deleted vertices $D_t$, and update $G_{t+1}:= G_t\setminus (V(H_t)\cup D_t)$. We repeat this process until we no longer can, ultimately yielding a maximal collection of subgraphs $H_1,\dots,H_\tau$. Noting $|H_i|\ge d^{5C}$ holds for all $i$ by our connectivity and girth assumptions, we get $\tau \le n/d^{5C}$. 
 For $i=0,\dots,\tau$, set $S^{i} := \bigcup_{t=1}^i S_t$ and $D^{ i}\coloneqq \bigcup_{t=1}^i D_t$.
If for some $i\leq \tau$, $H_i$ contains $500d^4$ vertices in $B$ with $d_{H_t}(x)=d_G(x)$ then we obtain the required subgraph. 

Write $\Delta:=\Delta(G)\le d^C$. We claim that for each $i=0,1,\dots,\tau$ that $|D^i|\leq \Delta|S^i|$. Indeed, by construction, the number of edges touching $D^i$ but not $S^i$ is at most $d^5|D^i|$, whence
\[d^6|D^i| \le \sum_{x\in D^i} d(x)\le 2d^5|D^i|+ \sum_{y\in S^i} d(y)\le (d^6-1)|D^i|+\Delta|S^i|.\]Consequently, noting $|S^\tau|\le 400 d^4 \tau\le n/d^{4C}$ we get that $|S^\tau \cup D^\tau|\le 2n/d^{3C}\le |B|/4d^{C}$ (since $|B|\ge n/d^C$ and $d^C\ge 8$). Thus in particular, writing $W$ to be the vertices with distance at most $1$ to $S^\tau \cup D^\tau$, we have $|W|\le 2\Delta|S^\tau\cup D^\tau|\le |B|/2$.

Now, since our sequence of graphs $H_1,\dots,H_\tau$ is maximal, we must have that $G_{\tau+1}$ is the empty graph. Thus, $V(G)\setminus (S^\tau\cup D^\tau)$ is partitioned into the vertex sets of $H_1,\dots,H_\tau$. Next let $B':= B\setminus W$, and note $|B'|\ge (1/2)|B|$ by the previous paragraph. Then by pigeonhole we can find some $t\in [\tau]$ with $|B'\cap V(H_t)|\ge \tau^{-1}|B'|\ge (1/2)d^{4C}$. We observe that $d_{H_t}(v)=d_G(v)$ for all $v\in B'$ (since $B'\cap W=\emptyset$, and considering the definition of $W$), whence taking $H':=H_t$ has the desired properties.
\end{proof}

\section{Main proofs}

\subsection{Key proposition}

We turn now to the main technical result of the paper, which allow us to find an induced subdivision provided the maximum degree is polynomially bounded (even with a weaker minimum degree assumption). 
\begin{lemma}\label{lem: maxdegree}
    Let $d\geq 10^7$ and $G$ be a graph on $n$ vertices and $U\subseteq V(G)$ where $|U|\geq n/20$ such that $d(x)\geq d$ for every $x\in U$.
    Suppose that $\delta(G)\geq d/5$ and $\Delta(G)\leq d^{35}$. Finally, assume that $girth(G)\geq 10^8$. Then, $G$ contains an induced subdivision of $K_{d+1}$. 
\end{lemma}

\begin{proof}
We proceed by finding a somewhat delicate structure. So let us first describe the structure and show why it suffices.

Specifically, we will find some set of vertices $S\subset V(G)$, and define an auxiliary graph $H$ on these vertices (i.e., $V(H)=S$). For each $e\in E(H)$, we shall associate some induced path $P_e$ of length at most $300$, in such a way that for disjoint $e,e'\in E(H)$, we have that $G[V(P_e)\cup V(P_{e'})]$ is isomorphic to the union of the two paths (i.e., the paths will be vertex disjoint, and send no edges between). 

In this graph $H$, we say a vertex $v$ is ``branchable'' if we can find $d$ neighbors $u_1,\dots,u_d\in N_H(v)$ so that the paths $G[V(P_{u_1v})\cup \dots \cup V(P_{u_dv})\setminus\{v\}]$ is isomorphic to the union of $d$ paths (meaning the induced subgraph including $v$ must be a subdivision of a star, whose leaves are $u_1,\dots,u_d$.   
\begin{claim}\label{clm: structure}
    Suppose we find such a structure $S,H,(P_e)_{e\in E(H)}$, so that $H$ is $11d^2$-connected, and contains at least $d+1$ branchable vertices $v_1,\dots,v_{d+1}$, $v_i$ and $v_j$ having at least three for all $i\neq j$. Then $G$ contains an induced subdivision of $K_{d+1}$.
\end{claim}
\begin{proof}
    For $i=1,\dots,d+1$, fix $u_1^{(i)},\dots,u^{(i)}_{d}$ witnessing that $v_i$ is branchable. Note that the vertices $v_1,\dots,v_{d+1},u_1^{(1)},\dots,u_d^{(d+1)}$ are all distinct, as otherwise we'd have two vertices $v_i,v_j$ with distance at most $2$. Now, after removing the vertices $v_1,\dots,v_{d+1}$ from $H$, the remaining graph is still $10d^2$-connected, and thus $d^2$-linked by Lemma~\ref{lem: linkedness}.
    
    Thus, we can find pairwise vertex-disjoint paths $(Q_{ij})_{ij\in \binom{[d+1]}{2}}$ inside $H\setminus\{v_1,\dots,v_{d+1}\}$, so that for $i<j$, $Q_{ij}$ has end-points $u^{(i)}_{j-1}$ and $u^{(j)}_i$. Now for each $ij$, writing $V_{ij}:= V(P_{v_iu^{(i)}_{j-1}}) \cup \bigcup_{e\in E(Q_{ij})}V(P_e) \cup V(P_{v_ju_i^{(j)}})$, this induces a connected graph in $G$, thus we can consider the shortest path $P_{ij}'$ from $v_i$ to $v_j$ inside $G[V_{ij}]$. We claim $G[\bigcup_{ij} V(P_{ij}')]$ induces a $K_{d+1}$-subdivision; clearly it suffices to check we contain no edges besides those in $\bigcup_{ij} E(P_{ij}')$.
    
    Suppose there was some undesired edge $xy$ in our induced subgraph. Then we must have $e_x,e_y\in \binom{[d+1]}{2}$ so that $x\in V(P_{e_x}'),y\in V(P_{e_y}')$. Note that each $P_{ij}'$ is an induced path inside $G$ (otherwise we could replace it with a shorter path using a subset of the vertices), thus we may suppose $e_x\neq e_y$. Next pick edges $f_x,f_y\in E(H)$ so that $x\in V(P_{f_x}),y\in V(P_{f_y})$. If $f_x$ and $f_y$ are vertex-disjoint, then by the assumptions of our structure $xy\not\in E(G)$, contradiction. Thus, we may assume $f_x=v_iu_{j_1}^{(i)},f_y=v_iu_{j_2}^{(i)}$ for some $i,j_1,j_2$. But then $xy$ cannot be an unwanted edge, since $v_i$ is branchable. Thus we conclude the subdivision is indeed induced, as desired.
\end{proof}
A convenient consequence of this is the following.
\begin{corollary}\label{cor: structure}
    Suppose we find such a structure $S,H,(P_e)_{e\in E(H)}$, so that $H$ has minimum degree $d^{6}$, and contains $|S|/d^{5000}$ branchable vertices. Then $G$ contains an induced $K_{d+1}$-subdivision.
\end{corollary}
\begin{proof}
    Recall that $G$ has maximum degree $d^{35}$, girth $10^8$, and the paths $P_e$ all have length at most $300$. Thus writing $C:=7000$, we have $\Delta(H)\le d^C$ and $H$ has girth at least $(1/2)10^6\ge 5\cdot C$. 
    
    Thus, writing $B\subset V(H)$ to denote the branchable vertices in $S$, we can apply Lemma~\ref{lem: connectedgood} to find a $100d^2$-connected subgraph $H'\subset H$, with at least $(1/2)d^{4C}$ vertices $b\in B\cap V(H')$ with $d_{H'}(b)=d_H(b)$ (implying they are still branchable in the structure $V(H'),H',(P_e)_{e\in E(H')}$). Since we have maximum degree $d^C$, we can greedily find $d+1$ such vertices which pairwise have distance at least $2$ from one another. Whence, $G$ will be forced to contain an induced $K_{d+1}$-subdivision by Claim~\ref{clm: structure}
\end{proof}

It remains to find the structure asked for in Corollary~\ref{cor: structure}.

Take a maximal set of vertices $U'\subseteq U$ which have pairwise distance more than $101$. Observe that $|P_1|\geq \frac{|P|}{7\Delta(G)^{101}}\geq \frac{n}{d^{3600}} $. We now take $W\subset V(G)$ maximal so that all vertices in $U'\cup W$ have pairwise distance at more than 101.

For each $x\in U'\cup W$, let $B^{(0)}(x)$ be the ball of radius $50$ with center $x$.
Observe that by construction all these balls consist of pairwise vertex-disjoint induced trees. Next, it is easy to see we may enlarge these balls to sets $B(x)\supset B^{(0)}(x)$, for $x\in U' \cup W$ so that every vertex of $G$ is in exactly one such ball and each such set $B(x)$ induces a tree where all vertices are at distance at most $101$ from $x$.

Let us create an auxiliary graph $H^*$ on $S^*:= U'\cup W$, where two vertices $x,y$ are adjacent if there is a path of $P_{xy}$ of length at most $203$ with end-points $x,y$, so that $V(P_{xy})\subset B(x)\cup B(y)$. Now, let us also create an auxiliary graph $F$ between $E(H^*)$ and $S^*$, where $e=uv\in E(H)$ is adjacent to $w$ if there is an edge between $V(P_{uv})$ and $B(w)$. 

We will now consider taking a random $S\subset S^*$, where each $x$ is kept with probability $p=d^{-36}$ independently. We will then take $H\subset H^*[S]$ to be the subgraph where an edge $e=xy$ is kept if $N_F(e)\cap S =\{x,y\}$ (meaning that if $e=xy,e'=x'y'\in E(H)$ are vertex-disjoint, there are no edges between $V(P_e)$ and $B(x')\cup B(y')\supset V(P_{e'})$). Thus, $S,H,(P_e)_{e\in E(H)}$ is the type of structure we wish to find, and it remains to show there is an outcome where $\delta(H)\ge d^6$ and we have at least $n/d^{5000}\ge |S|/d^{5000}$ branchable vertices. Indeed, upon establishing this, we get our induced subdivision by Corollary~\ref{cor: structure}. 

We will achieve this via the local lemma. 
\begin{claim}\label{clm: rare event}
    Fix $x\in S^*$ and $z\in N_G(x)$. Let $\mathcal{E}_{xz}$ be the event that $x\in S$, and furthermore that there are fewer than $d^6$ paths $P_{xy}$ with $z\in V(P_{xy})$ so that $xy\in E(H)$.
    
    Then we have $\mathbb{P}(\mathcal{E}_{xz})\le \exp(-d)$.
\end{claim}
Before establishing this final claim, let us explain why we are done. For each $x\in S^*$ and $z\in N_G(x)$, we consider the event $\mathcal{E}_{xz}$ as defined above. Each such $\mathcal{E}_{xz}$ depends on whether at most $\Delta(H)\cdot 300\Delta(G)+1\le d^{7500}$ vertices are included in $S$ or not. And each vertex $x\in S^*$ influences at most $d^{7500}$ events. So, the dependence graph has maximum degree at most $d^{15000}\lll \exp(d)$, and thus $H$ will satisfy all these events with probability at least $(1-e\exp(-d))^{n\Delta(G)}\ge \exp(-nd^{36}\exp(-d))\ge \exp(-n\exp(-d/2))$ (by Lemma~\ref{lem: lll}). Furthermore, by a Chernoff bound, we have that $|U'\cap S|\ge (p/2)|U'|$ happens with probability at least $1-\exp(-p|U'|/8)\ge 1-\exp(-nd^{-50})>1-\exp(-n\exp(-d/2))$. Thus, we can fix an outcome $H$ where we have both properties.Whence $\delta(H)\ge d^6$ clearly holds. Moreover, we noting each vertex in $U'\cap S$ is branchable\footnote{We include a brief explanation, for the convenience of the reader (though we think it is more informative to your own proof of this claim). 

Since $x\in U'$, it has $d$ neighbors $z_1,\dots,z_d$. By assumption, none of the events $\mathcal{E}_{xz_i}$ held, thus we can find $e_1,\dots,e_d\in E(H)$ so that $z_i\in V(P_{e_i})$. Since $x$ is an endpoint of all these paths, and we have girth $>500$, it must follow that these paths witness that $x$ is indeed branchable (since any two $z_i,z_j$ have distance at least $500$ in $G\setminus\{x\}$).}, which will establish we have at least $(p/2)|U'|\ge n/d^{5000}$ vertices; meaning $H$ has the needed additional properties to apply Corollary~\ref{cor: structure} and complete our proof.

Without further ado, let us show that these events $\mathcal{E}_{xz}$ are easy to avoid.
\begin{proof}{Proof of Claim~\ref{clm: rare event}}
    By assumption, some $x\in S^*,z\in N_G(x)$ is fixed. We will show that whenever $x\in S$, that it is very likely to keep many edges in $H$ corresponding to paths passing through $z$.

    To this end, we will show there exists $k=d^{45}$ edges $e_1,\dots,e_k$ so that $z\in V(P_{e_i})$ for all $i\in [k]$, and $N_F(e_i)\cap N_F(e_j) =\{x\}$ for any $j\neq i$. Whence, it follows that, after conditioning on $x\in S$, the events $\{e_i\in E(H)\}$ are independent.

    Noting that $\mathbb{P}(e_i\in E(H)|x\in S) \ge p(1-p)^{300 \Delta(G)}\ge p/2$ (recall $p=d^{-36}<\frac{1/2}{300\Delta(G)})
    $, we get that $\mathbb{P}(\mathcal{E}_{xz})\le \mathbb{P}(\operatorname{Bin}(k,p/2)<d^6)$. But noting $k(p/2)\ge 2 d^6$, a Chernoff bound then implies $\mathbb{P}(\mathcal{E}_{xz})\le \exp(-k(p/16))\le \exp(-d^6/8)\le \exp(-d)$, as desired.

    So it remains to find these edges $e_1,\dots,e_k\in E(H^*)$. We observe that given two paths $P_e,P_{e'}$ sharing $x$ as an end-point, that if $N_F(e)\cap N_F(e')$ contains some vertex $w\neq x$, there there must be some $v\in V(P_e)\cap V(P_{e'})$ so that $N_G(v)\cap B(w)$ is non-empty. Indeed, if this is not the case, there are two distinct vertices with distance less than $300$ to both $x$ and $w$, which will contradict our girth assumption. 

    Now, we consider the set $L_z$ of leaves in $G[B_0(x)]$, whose path to $x$ passes through $z$. We claim that for each $\ell\in L_z$, that there is some $y_\ell\in N_{H^*}(x)$ so that the path $P_{xy_\ell}$ passes through $\ell$ (and thus also $z$). Upon confirming this, we will be done. Indeed, given $\ell,\ell'$, we'd have $V(P_{xy_\ell})\cap V(P_{xy_{\ell'}})$ would only contain vertices with distance at most $49$ to $x$ (by girth). Thus the neighborhoods of those vertices would be contained in $B_0(x)\subset B(x)$ (meaning $N_F(xy_\ell)\cap N_F(xy_{\ell'})=\{x\} $ as desired). Meanwhile, since $G$ had minimum degree $d/5$, one gets $|L_z|\ge (d/5-1)^{49}\ge k$, so we get enough edges for the claim to follow.

    To see that each $\ell$ has a $y_\ell$, we recall that $G[B(x)]$ is a tree with radius at most $101$. Thus, there is some leaf $l_\ell$ in that tree, whose path to $x$ contains $\ell$. Then taking any neighbor of $l_\ell$ not lying in $B(x)$, said neighbor must lie in $B(y_\ell)$ for some $y_\ell\in S^*$. Thus taking the path from $x$ to $l_\ell$ to $y_\ell$ has length at most $203\le 300$, and is contained in $B(x)\cup B(y_\ell)$ (implying $xy_\ell\in E(H^*)$ as required). 
\end{proof}
\end{proof}

\subsection{Proof of theorem}
We are now ready to prove our main Theorem~\ref{thm: main}

\begin{proof}[Proof of Theorem~\ref{thm: main}]
Let $G$ be a graph on $n$ vertices with girth at least $10^5$ and minimum degree at least $d\geq 10^8$. By taking $G'\subseteq G $ with highest minimum degree, we may and will assume $G$ is $(d+1)$-degenerate. Letting $B\subset V(G)$ be the set of vertices of degree at least $d^{35}$, we note $|B|\leq \frac{100n}{d^{34}}$ (due to degeneracy).

Let $A\subseteq V(G)\setminus B$ be the set of vertices which send at least two edges to $B$. By Lemma~\ref{lem: unbalanced}, we must have $|A|\leq \frac{1000n}{d^{15}}$, otherwise we are done. Then, let $A'\coloneqq \{y\in V(G)\setminus B: d_B(y)=1\}$. 
We will consider two cases, based off of whether or not $|A'|$ is large.

\subsection{\texorpdfstring{\textbf{Case $1$}. $|A'|\geq n/2$.}{Case 1: |A'| >= n/2}} 

By Lemma~\ref{lem: largesub} with $X:=A'$, we may find $X'\subseteq A'$ and $Y\subseteq (V(G)\setminus X')$ such that:
\begin{itemize}
    \item $|Y|\ge n/(3d^7)$;
    \item $Y$ is an independent set;
    \item $2\leq|N(y)\cap X'|\leq d^6$, for all $y\in Y$;
    \item $d(x')\le 8d$ for all $x'\in X'$.
\end{itemize}Since the size of $A\cup B$ is small we have that $|Y\setminus (A\cup B)|\geq \frac{n}{4d^7}$. For simplicity of notation $Y\coloneqq Y\setminus (A\cup B)$.

We are now going to find an induced subdivision of $K_{d+1}$ with branch vertices in $B$.
First, observe that for every $y\in Y$, there are at most $|N(y)\cap X'|8d\le 8d^7$ other $y'\in Y$ with $N(y)\cap N(y')\cap X'\neq \emptyset$. Thus, we may take $Y'\subseteq Y$ of size $\frac{|Y|}{8d^7+1}\ge \frac{n}{100d^{14}}$, all with disjoint neighborhoods into $X'$.

Fix a $2d$-degenerate ordering of $B$ and a $2d$-degenerate ordering of $X'$ (thus every vertex has at most $2d$ ``right-neighbours'' in this ordering). We now take each vertex in $X'\cup B$ with probability $p\coloneqq\frac{1}{4d^{6}}$ (independently). Let the random subset of $B$ be denoted by $R_B^{(0)}$ and the random subset of $X'$ be $R_{X'}^{(0)}$. Then let $R_B\subset R_B^{(0)}$ be the vertices which also have no right-neighbours in $R_B^{(0)}$.

Our definition of $R_{X'}$ is a bit more involved. By assumption, for every $x\in X'$, there is a unique $b(x)\in B$ where $b(x)\in N(x)\cap B$. We now let $R_{X'}\subset R_{X'}^{(0)}$ be the set of $x\in R_{X'}^{(0)}$ with $b(x)\in R_B$ \textit{and} with zero right-neighbours in $R_{X'}^{(0)}$.

Over the randomness of $R_B$ and $R_{X'}$, we say that $y\in Y'$ is \textit{good} if $|N(y)\cap R_{X'}|=2$ and $N(y)\cap R_B =\emptyset$. Let $Y''\subset Y'$ be the set of good vertices.
\begin{claim}
     We have $\mathbb{E}[|Y''|]\geq \frac{n}{10^5d^{31}}$.
 \end{claim}
 \begin{proof}
    Fix $y\in Y'\subset Y$. By assumption, we have $2\le |N(y)\cap X'|\le d^6$, thus we can fix distinct $x_1,x_2\in N(y)\cap X'$. Also, we have $|N(y)\cap B|\le 1$ since $Y\cap A=\emptyset$. It then follows that
     \[\mathbb{P}[ y \text{ is good}]\geq p^2(1-p)^{d^6}\cdot (1-p)^{2d}(1-p)^{2d} \cdot p^2(1-p)^{2d}(1-p)^{2d}\geq p^4/4=\frac{1}{1024d^{24}},
     \]with the first term lower-bounding the probability $N(y)\cap (R_{X'}\cup R_B)=\{x_1,x_2\}$, the second term lower-bounding the probability $x_1,x_2$ have no right-neighbours in $R_{X'}^{(0)}$ (importantly $x_1x_2\not\in E(G)$ by our girth assumption), and the third term lower-bounding the probability $b(x_1),b(x_2)\subset R_B$ (importantly $b(x_1)$ and $b(x_2)$ do not belong to $N(y)\cap B$, as we do not have triangles).

Hence, the expected number of good vertices is at least $\frac{|Y'|}{1024d^{24}}\geq \frac{n}{10^5 d^{31}} $. 
 \end{proof}
 
Fix a choice of $R_{B}, R_{X'}$ with at least $\frac{n}{10^5 d^{31}}$ good vertices. 
We consider now an auxiliary graph $H$ on the vertex set $R_B$, where we add an edge between $z_1,z_2\in R_B$ if there is a good vertex $y\in Y''$ whose two unique neighbours in $R_{X'}$, say $x,w$ satisfy $b(x)=z_1$ and $b(w)=z_2$. Note that each good vertex $y$ contributes with a distinct auxiliary edge, since our graph $G$ has no cycle of length smaller than $9$. As $|Y''|\geq 10d^2|B|$ (recall $|B|\le \frac{100 n}{d^{34}} \le  \frac{10^7}{d^3}|Y''|$ and that we assume $d\ge 10^8$), Lemma~\ref{thm: subdivision} guarantees $H$ has a subdivision of a $K_{d+1}$. Noting that for any subgraph $H'$ of $H$, the $3$-subdivision of $H$ is an induced subgraph of $G$,\footnote{Here, we highlight that each edge $e\in E(H)$ corresponds to three vertices $y,x_1,x_2$ which all cannot belong to $R_B$. Indeed, $Y\cap B=\emptyset$ implies $y\not\in R_B\subset B$, while $|N(x_1)\cap R_B|=|N(x_2)\cap R_B|=1$ (meaning neither vertex belongs to the independent set $R_B$). The rest quickly follows recalling $R_B,R_{X'},Y''$ are all independent sets of $G$.} we see $G$ indeed contains an induced subdivision of $K_{d+1}$ as desired. This concludes the analysis of Case $1$. 
 
\subsection{\texorpdfstring{\textbf{Case $2$.} $|A'|\leq n/2$.}{Case 2: |A'| <= n/2}}

Let $G'$ be the induced subgraph on $V(G)\setminus B$. Note that $|V(G')| \geq \frac{4n}{5}$ and all but at most $n/2+|A'|$ vertices have degree at least $d$. In particular, $G'$ has at least $n/3$ vertices $x$ with $d_{G'}(x)\ge d$.

We now iteratively delete the vertices of $G'$ having than $d/5$ neighbors remaining in $V(G')$ (essentially in the fashion as in Lemma~\ref{lem: connectedgood}). Formally, let $G'_0:=G'$ and for $t\ge 0$, if $G'_t$ has a vertex $v_t$ with $d_{G'_t}(v_t)<d/5$ update $G'_{t+1}:= G'_t\setminus \{v_t\}$ and continue. Let $S=\{v_1,\dots,v_\tau\}$ denote the set of vertices deleted by this process. We note this cannot be too large.

\begin{claim}
    $|S|< 6\cdot 10^5 d^4|A|$.
\end{claim}
\begin{proof}
     Suppose not, meaning that $|S|\ge 6\cdot 10^5d^{4}|A|$. We seek to derive a contradiction, using the facts that each vertex $x\in S\setminus A$ has degree at least $d-1$ in $G'$ and hence sends many edges inside $S$, while by construction $e(G'[S\setminus A])\leq (d/5) |S\setminus A|$.

     Letting $S'\subset S\setminus A$ be the set of vertices $x$ with $|N(x)\cap A|\le 1$, we have that that $|N(x)\cap S|\ge (4/5)d-2\ge (2/3)d$. Thus $(2/3)d|S'|\le e(G'[S\setminus A])\ge (1/5)d|S|$, implies $|S'|<\frac{3}{10}|S|<(1/3)|S|$.

     Meanwhile, $|S|\le \frac{1}{6}|A|$ by assumption, so we must have at least $|S\setminus (A\cup S')|\ge |S|/2\ge 10^5d^4|A|$ vertices $x\in S\setminus A$ with $|N(x)\cap A|\ge 2$. But then, we may apply Lemma~\ref{lem: unbalanced} (with $A:= A,B:= \{x\in S\setminus A:|N(x)\cap A|\ge 2|\}$) to conclude that $G$ contains an induced subdivision of $K_{d+1}$ as desired. 
\end{proof}

Let us denote the remaining graph $G'\setminus S$ by $G''$. We observe it has the following properties:
\begin{enumerate}[(i)]
    \item $\Delta(G'')\leq d^{35}$;
    \item $\delta(G'')\geq d/5$;
    \item $|\{x\in V(G''):d_{G''}(x)\ge d\}|\geq \frac{n}{6}$;
    \item $girth(G'')\geq 10^5$.
\end{enumerate}

Observe that $(i)$ and $(ii)$ follow by construction of $G''$. Moreover, $(iii)$ follows from the fact $G'$ had at least $n/3$ vertices with degree $d$, and at most $(d/5)|S|<n/6$ of them have a neighbor in $S$ (leaving $n/6$ vertices $x$ with $d_{G''}(x)=d_{G'}(x)\ge d$). Finally, the girth of $G''$ is at least the girth of $G$.

We may now apply Lemma~\ref{lem: maxdegree} to $G''$ to obtain an induced copy of $K_{d+1}$, as required. This concludes Case 2 and proves our theorem.
\end{proof}

\bibliographystyle{plainnat}
\bibliography{references}

@article{BollobasThomason1998TopologicalCompleteSubgraphs,
  author  = {Bollob{\'a}s, B. and Thomason, A.},
  title   = {Proof of a Conjecture of {M}ader, {E}rd{\H{o}}s and {H}ajnal on Topological Complete Subgraphs},
  journal = {European Journal of Combinatorics},
  volume  = {19},
  number  = {8},
  pages   = {883--887},
  year    = {1998},
  doi     = {10.1006/eujc.1997.0188}
}

@article{KomlosSzemeredi1994TopologicalCliques,
  author  = {Koml{\'o}s, J. and Szemer{\'e}di, E.},
  title   = {Topological Cliques in Graphs},
  journal = {Combinatorics, Probability and Computing},
  volume  = {3},
  number  = {2},
  pages   = {247--256},
  year    = {1994},
  doi     = {10.1017/S0963548300001140}
}

@article{Mader,
  author  = {Mader, W.},
  title   = {Topological Subgraphs in Graphs of Large Girth},
  journal = {Combinatorica},
  volume  = {18},
  number  = {3},
  pages   = {405--412},
  year    = {1998},
  month   = mar,
  doi     = {10.1007/PL00009829},
  url     = {https://doi.org/10.1007/PL00009829},
  issn    = {1439-6912}
}

@article{LiuMontgomery2017MaderC4Free,
  author  = {Liu, H. and Montgomery, R.},
  title   = {A proof of {M}ader's conjecture on large clique subdivisions in ${C}_4$-free graphs},
  journal = {Journal of the London Mathematical Society},
  series  = {2},
  volume  = {95},
  number  = {1},
  pages   = {203--222},
  year    = {2017},
  doi     = {10.1112/jlms.12019},
  eprint  = {1605.07791},
  archivePrefix = {arXiv}
}

@incollection{ErdosLovasz1975LLL,
  author    = {Erd{\H{o}}s, P. and Lov{\'a}sz, L.},
  title     = {Problems and results on 3-chromatic hypergraphs and some related questions},
  booktitle = {Infinite and Finite Sets (Colloq., Keszthely, 1973), Vol. II},
  series    = {Colloq. Math. Soc. J{\'a}nos Bolyai},
  volume    = {10},
  pages     = {609--627},
  publisher = {North-Holland},
  address   = {Amsterdam},
  year      = {1975}
}

@book{AlonSpencerProbMethod,
  author    = {Alon, N. and Spencer, J. H.},
  title     = {The Probabilistic Method},
  publisher = {Wiley},
  edition   = {3},
  year      = {2008}
}

@article{AlonHooryLinial2002MooreIrregular,
  author  = {Alon, N. and Hoory, S. and Linial, N.},
  title   = {The {M}oore bound for irregular graphs},
  journal = {Graphs and Combinatorics},
  volume  = {18},
  number  = {1},
  pages   = {53--57},
  year    = {2002},
  doi     = {10.1007/s003730200002}
}

@article{PenevThomasseTrotignon2014Isolating,
  author  = {Penev, I. and Thomass{\'e}, S. and Trotignon, N.},
  title   = {Isolating highly connected induced subgraphs},
  journal = {arXiv preprint arXiv:1406.1671},
  year    = {2014}
}

@article{KuhnOsthus2006ImprovedBounds,
  author  = {K{\"u}hn., D. and Osthus, D.},
  title   = {Improved bounds for topological cliques in graphs of large girth},
  journal = {SIAM Journal on Discrete Mathematics},
  volume  = {20},
  number  = {1},
  pages   = {62--78},
  year    = {2006},
  doi     = {10.1137/040617765}
}

@article{WollanThomas,
  author  = {Thomas, R. and Wollan, P.},
  title   = {An improved extremal function for graph linkages},
  journal = {European J. Combin.},
  volume  = {26},
  pages   = {309--324},
  year    = {2005},

}

\end{document}